\documentclass[12pt]{article}

\usepackage{graphicx}
\usepackage{amsmath,amsthm,amsfonts,amscd,amssymb,comment,eucal,latexsym,mathrsfs}
\usepackage{stmaryrd}
\usepackage[all]{xy}

\usepackage{epsfig}

\usepackage[all]{xy}
\xyoption{poly}
\usepackage{fancyhdr}
\usepackage{wrapfig}
\usepackage{epsfig}
\usepackage{xcolor}
\usepackage{hyperref}
\hypersetup{
  colorlinks=true,
  linkcolor=blue,
  citecolor=red
}

\theoremstyle{plain}

\newtheorem{thm}{Theorem}[section]

\newtheorem{lem}[thm]{Lemma}
\newtheorem{cor}[thm]{Corollary}

\newtheorem{claim}[thm]{Claim}
\theoremstyle{definition}
\newtheorem{defn}{Definition}
\theoremstyle{remark}
\newtheorem{remark}{Remark}

\advance \topmargin by -\headheight
\advance \topmargin by -\headsep
\evensidemargin \oddsidemargin
\def\det{{\textrm{det}}}

\def\ker{{Ker}}

\def\N{{\mathbb N}}

\def\P{{\cal P}}

\def\R{{\mathbb R}}

\def\Sub{\textrm{Sub}}

\def\supp{{\textrm{supp}}}

\DeclareMathOperator{\tr}{\mathrm{tr}}

\def\chix{{\raise.5ex\hbox{$\chi$}}}

\def\Z{{\mathbb Z}}

\def\SL{\mathrm{SL}}
\def\GL{\mathrm{GL}}
\def\slz{\mathrm{SL}_d(\Z)}
\def\slr{\mathrm{SL}_d(\R)}
\def\subg{\mathrm{Sub}_G}
\def\dd{\mathrm{d}}

\def\BH{\mathbb{H}}

\renewcommand{\P}[1]{{\mathbb{P}}\left[{#1}\right]}

\newcommand\Aut{\operatorname{Aut}}
\newcommand\pr{\operatorname{pr}}

\begin{document}
\title{Invariant random subgroups of semidirect products}

\author{Ian Biringer\footnote{Boston College. Supported in part by NSF
    grant DMS-1611851 and CAREER Award DMS-1654114.}, Lewis
  Bowen\footnote{University of Texas at Austin. Supported in part by
    NSF grant DMS-0968762, NSF CAREER Award DMS-0954606 and BSF grant
    2008274.} and Omer Tamuz\footnote{California Institute of
    Technology. This work was supported by a grant from the Simons
    Foundation (\#419427, Omer Tamuz).}}
  
\maketitle

\begin{abstract}
We study invariant random subgroups (IRSs) of semidirect products $G = A \rtimes \Gamma$. In particular, we characterize all IRSs of parabolic subgroups of $\slr$, and show that all ergodic IRSs of $\R^d \rtimes \slr$ are either of the form $\R^d \rtimes K$ for some IRS of $\slr$, or are induced from IRSs of $\Lambda \rtimes \mathrm{SL}(\Lambda)$, where $\Lambda < \R^d$  is a lattice.
\end{abstract}

\tableofcontents

\section{Introduction}

Let $G $ be a locally compact, second countable group and let $\subg$
be the space of closed subgroups of $G $, considered with the Chabauty
topology~\cite{Chabautylimite}.

\begin {defn}An \emph {invariant random subgroup} (IRS) of $G$ is a random element of $\subg$ whose law is a conjugation invariant Borel probability measure. 	
\end {defn}

IRSs were introduced by Ab\'ert--Glasner--Vir\'ag in
\cite{Abertkesten}, and independently by Vershik~\cite{Vershiktotally} (under a different name) and the second author \cite{bowen2010random}. Examples of IRSs include
normal subgroups, as well as random conjugates $g\Gamma g^{-1}$ of a
lattice $\Gamma < G$, where the conjugate is chosen by selecting
$\Gamma g$ randomly against the given finite measure on
$\Gamma \backslash G$. More generally, any IRS of a lattice $\Lambda < G$ \emph {induces} an IRS of $G$: if $\mu_\Gamma$ is the law of the original IRS and $\eta$  is a $G $-invariant probability measure on $ \Gamma \backslash G$, the new law $\mu_G$  is given by the integral $$\mu_G = \int_{\Gamma g  \in \Gamma \backslash G} g_* \mu_\Gamma \ d\eta,$$
where $\mu_\Gamma$ is regarded as a  measure on $\mathrm{Sub}_\Gamma \subset \subg$, and $g$ acts on $\subg$  by conjugation.  Informally, we conjugate the IRS of $\Gamma$ by an `$\eta$-random' element of $G$. 
 Since $\subg$ is compact \cite[Lemma
E.1.1]{Benedettilectures}, the space of (conjugation invariant) Borel
probability measures on $\subg$ is weak* compact, by Riesz's 
representation theorem and Alaoglu's theorem. Hence, IRSs compactify
the set of lattices in $G$. There is a growing literature on IRSs
(see, e.g.,~\cite{bader2014amenable,biringer2017unimodularity,bowen2012invariant,bowen2015invariant, thomas2014invariant}) and their
 applications, see especially \cite{
  abert2012growth, bowen2010random, hartman2015furstenberg, tucker2015weak}.

Our goal in this note is to develop an understanding of IRSs of semidirect
products $G=A \rtimes \Gamma$.  There are few
general constructions of such IRSs: there is the trivial IRS
$\{e\}$, and IRSs of the form $A \rtimes K$, where $K$ is an IRS of
$\Gamma$. When the kernel $\Gamma_{triv}$ of the action
$\Gamma \circlearrowright A$ is nontrivial, one can also construct
IRSs of the form $H \rtimes K$, where $H$ is an IRS of $A$ and $K$ is
an IRS of $\Gamma$ that lies in $\Gamma_{triv}$, but additional
examples are hard to find.

The kernel of our work are Theorems \ref{thm:general-irs1} and \ref{thm:general-irs2}, in which we study `transverse' IRSs of $G=A \rtimes \Gamma$ when $A$ is torsion-free abelian or simply connected nilpotent.  Here, an IRS $H<G$ is \emph{transverse} if $H \cap A = \{0\}$. This theorem has two parts: when $A$ is torsion-free abelian, we prove that that the projection of $ H$ to $\Gamma$ acts trivially on $A$  almost surely, and if $A$ is a simply connected nilpotent Lie group, we show that an (often large) subgroup of $\Gamma$ acts precompactly on the Zariski closure of the set of all first coordinates of elements $(v,M)\in H$, as $H$ ranges through the support of the IRS.

As applications of Theorems \ref{thm:general-irs1} and \ref{thm:general-irs2}, we study IRSs of two  familiar semidirect products: the \emph{special affine groups} $\R^d \rtimes \slr$ and the \emph {parabolic subgroups} of $\slr$.

\subsection{IRSs of special affine groups}
We are particularly interested in IRSs of $\R^d \rtimes \slr$.  In addition to the
examples $\{e\}$ and $\R^d \rtimes K$ mentioned above, one can construct an IRS from a lattice $\Lambda \subset \R^d$. Namely, the subgroup $\SL(\Lambda) < \slr$ stabilizing $\Lambda$ is also a lattice, see \cite{Raghunathandiscrete}, so the semidirect product $\Lambda \rtimes \SL(\Lambda) $ is a lattice in $\R^d \rtimes \slr$, and hence a random conjugate of it is an IRS.

\begin{thm}
  \label{thm:slr-irs}
  Let $H$ be a non-trivial ergodic IRS of $\R^d \rtimes \slr$. Then
  either 
\begin {enumerate}
\item $H = \R^d \rtimes K$  for some IRS $K < \slr$, or 
\item $H$  is induced from an IRS of $\Lambda \rtimes \SL(\Lambda)$, for some lattice $\Lambda <\R^d$.
\end {enumerate}
\end{thm}

Here, an IRS is \emph{ergodic} if its law is an ergodic measure for
the conjugation action of $G$ on $  \subg$.  By Choquet's theorem \cite{Phelpslectures}, every IRS can
be written as an integral of ergodic IRSs. Note that by transitivity of the action of $\slr$ on the space of lattices of a fixed covolume, we can actually choose $\Lambda$ in 2.\ to be a scalar multiple of $\Z^d$.

As a corollary, any normal subgroup of $\R^d \rtimes \slr$  is of the form $\R^d \rtimes K$ where $K$ is a normal subgroup of $\slr$. (Here, $K=\{e\},\, \slr$ or $\{\pm I\}$, where the last option is only available when $d$ is even.)   Similarly,  it follows that every lattice of $\R^d  \rtimes \slr$ is a finite index subgroup of some $\Lambda \rtimes \SL(\Lambda)$.  We expect that these results are not  entirely surprising, although we note that Theorem 4.8 of \cite{Diattalattices} is that $\R^d \rtimes \slr$  has no uniform lattices, which follows trivially from  this classification.

Stuck--Zimmer~\cite{stuck1994stabilizers} show that for $d > 2$, every
ergodic IRS of $\slr$ is either a lattice or a normal subgroup. This
result, together with Theorem~\ref{thm:slr-irs}, implies that for
$d > 2$ every ergodic IRS of $\R^d \rtimes \slr$ is likewise either a
lattice or a normal subgroup.

In light of  Theorem \ref{thm:slr-irs}, to understand IRSs in special affine groups it suffices to study those of $G = \Z^d \rtimes \slz$. 
There are the usual examples $\{e\}$ and $\Z^d \rtimes K$, where $K$ is an IRS of $\slz$, but in general, some subtle finite group theory appears.  For instance, let $$\pi_n : G \longrightarrow (\Z/n\Z)^d \rtimes \mathrm{SL}_d(\Z/n\Z)$$ be  the reduction map and setting $d=2$, consider the subgroup 
$$H = \left \{ \big ( (t,0),\left (\begin {smallmatrix} 1 & 1 \\ 0 & 1 \end{smallmatrix} \right )^t \big )  \ \Big |\  t \in \Z/n\Z \right \} <  (\Z/n\Z)^d \rtimes \mathrm{SL}_d(\Z/n\Z).$$
The preimage $\pi_n^{-1}(H)$ is a finite index subgroup of $G$,  and therefore can be considered as an IRS, but it does not have the form $\Lambda \rtimes K$  for any $\Lambda < \Z^d,K < \slz$.  However, we will show that all IRSs of $G$  are semidirect products up to some `finite index noise'.  Namely, let $$G_n = \ker \, \pi_n = n\Z^d \rtimes \Gamma(n),$$ where $\Gamma(n)$  is the   kernel of  the reduction map $\slz \to \mathrm{SL}_d(\Z/n\Z)$. We prove:


\begin{thm}
  \label{thm:affine-irss}
  Let $H$ be a non-trivial ergodic IRS of $\Z^d \rtimes \slz$. Then
  there is some $n \in \N$ such that $H_n = H \cap G_n$ is of the form
  $n\Z^d \rtimes K$, where $K$ is an IRS of $\slz$.
\end{thm}

\begin{remark}\label {rmk}
 In the case of $G=\SL_n(\Z)$, the Nevo--Stuck--Zimmer theorem says that any ergodic IRS of $G$ is either finite index  almost surely, or is central in $G$, see \cite{stuck1994stabilizers,nevo2002generalization}. Bekka \cite{bekka2007operator} later generalized this to a rigidity statement about the \emph{characters} of $G$. Here, an IRS with law $\mu$ gives the character
$\phi : G \longrightarrow [0,1],$ where  $\phi(g) = \mu( \{H \in \Sub_G  \ | \ g \in H \} ).$

Specializing Bekka's proof to the case of IRSs, Theorem \ref{thm:affine-irss} can be used in place of his Sections 4 and 5 (and a bit of 6) in a fairly elementary proof of the Nevo--Stuck--Zimmer theorem for $G=\SL_n(\Z)$. Namely, suppose $ H \leq G$ is an ergodic infinite index IRS. Writing
$$\Z^n = \langle x_1\rangle \oplus \cdots \oplus \langle x_n \rangle ,$$
we can let $P_i \cong \Z^{n-1} \rtimes \SL_{n-1}(\Z) $ be the parabolic subgroup of $G$ that is the stabilizer of $x_i$, and let $V_i \cong \Z^{n-1}$ be the corresponding unipotent subgroup of $P_i$. Theorem \ref{thm:affine-irss} says that for each $i$, either  $H \cap P_i$ is almost surely trivial or $H$ almost surely contains a lattice in $V_i$. If for every $i$, we have that $H \cap V_i $ is a lattice in $V_i$ a.s., then there is some $m$ such that a.s.\ $H$ contains the $m^{th}$ powers of all  elementary matrices, which implies $H$ is finite index, e.g.\ by Tits \cite{tits1976systemes}. So, we can assume that for some $i$, $H \cap P_i $ is trivial. Similarly, we can assume that $H \cap P_j^t$ is trivial for some $j$, where $P_j^t$ is the parabolic subgroup one gets by taking the transposes of all the matrices in $P_j$. Moreover, we can assume $i \neq j$, since if $P_i$ and $P_i^t$ were the only parabolics intersecting $H$ trivially, one would still get all possible $m^{th}$ powers of elementary matrices in $H $ as above. Switching indices  so that $(i,j)=(n,1)$ puts us at the beginning of Bekka's Section 6---and in fact, we already know Lemma 15. 

This gives a proof of the Nevo--Stuck--Zimmer theorem for $\SL_n(\Z)$  in which the only ingredients are our Theorem \ref{thm:affine-irss} (which is actually easier to prove than much of the content of this paper), the fact that the $m^{th}$ powers of all the elementary matrices generate a finite index subgroup of $\SL_n(\Z)$, and the  elementary arguments in \cite[\S 6]{bekka2007operator}. \end{remark}

\subsection{IRSs of parabolic subgroups of $\slr$}

 Suppose that $W=\R^d$ is a finite dimensional real vector space, written as a direct sum 
$$W=S_1 \oplus \cdots \oplus S_n$$
 of subspaces, and that $\mathcal{F}$  is the associated flag $$0=W_0 < W_1 < \cdots < W_n = W, \ \ \ W_k = \oplus_{i=1}^k S_i.$$  

Let $P < \SL(W)$ be the  corresponding \emph{parabolic subgroup}, i.e.\ the stabilizer of the flag $\mathcal F$, and let $V < P$    be the associated \emph {unipotent subgroup},  consisting of all $A \in P$  that act trivially on each of the factors $W_i/W_{i-1}$. We then have
$$P = V \rtimes R, \ \ R = \left \{ (A_1,\ldots,A_n) \in \prod_{i=1}^n \GL(S_i) \ \Big | \ \prod_i \det A_i = 1\right \}.$$
Elements of $P$ can be considered as upper triangular $n \times n$-matrices, where the $ij^{th}$ entry is an element of $\mathcal L(S_j,S_i)$,  the vector space of linear maps $S_j \longrightarrow S_i$. Elements of $R$ are diagonal matrices, and elements of $V$ are upper unitriangular.
 
Take a subset $\mathcal E \subset \{1, \ldots, n\}^2$  consisting of
pairs $(i,j)$ with $i<j$ and such that if $(i,j) \in \mathcal E$, then
$(i',j),(i,j') \in \mathcal E$ for $i'<i$ and $j'>j$. So, imagining
elements of $\mathcal E$ as corresponding to matrix entries, we are
considering subsets of entries above the diagonal, that are closed
under `going up' and `going to the right'.  Let $V_{\mathcal E}< P$  be
the normal subgroup consisting of all matrices that are equal to the
identity matrix except at entries corresponding to elements of
$\mathcal E$, and let $\mathcal K_{\mathcal E} < R$ be the kernel  of
the $R $-action (by conjugation) on $V/V_{\mathcal E}$.

\begin {thm}[IRSs of parabolic subgroups]\label {parabolictheorem}
The ergodic IRSs of $P$ are exactly the random subgroups of the form $V_{\mathcal E}\rtimes K$, where $K$ is an ergodic IRS of $\mathcal K_{\mathcal E}$.
\end {thm}

The subgroups $V_{\mathcal E}$ above are exactly the normal subgroups of $P$ that lie in $V$. So, a special case of the theorem is that an ergodic IRS of $P$  that is contained in $V$ is a normal subgroup of $P$. In fact, when proving Theorem \ref{parabolictheorem}, one first proves this special case, and then applies it to $H\cap V$ when $H$ is a general ergodic IRS of $P$.   Once one knows $H\cap V= V_{\mathcal E}$, the statement of Theorem \ref{parabolictheorem} is not a surprise, since  the only obvious way to construct an IRS $H$ with $H\cap V= V_{\mathcal E}$ is to take a semidirect product with an IRS of $\mathcal K_{\mathcal E}$.

The group $\mathcal K_{\mathcal E}$ can be described explicitly via matrices. Let $\mathcal I$ be the set of all $i\in \{1,\ldots,n\}$ such that if $i < n$, then $(i,i+1)\in \mathcal E$, and if $i >1 $, then $(i-1,i)\in\mathcal E$.  
Then $(A_1,\ldots,A_n)$ acts trivially on $V/V_{\mathcal E}$  exactly when for each maximal interval $\{i,\ldots,j\}\subset \{1,\ldots,n\} \setminus \mathcal I$, there is some $\lambda \in \R\smallsetminus \{0\} $  such that $A_{i}=\cdots=A_{j}=\lambda I.$ In a picture, if $\mathcal E$  consists of the starred entries below, then $(A_1,\ldots,A_n) \in \mathcal K_{\mathcal E}$  can be any diagonal matrix with the diagonal entries below,  subject to the additional condition $\prod_i \det A_i = 1$.
\begin {equation}\label {bigmatrix}\begin{pmatrix}
\lambda I & 0 &  \star & \star  & \star  & \star &  \star  & \star  \\
0 & \lambda I & \star & \star  & \star  & \star &  \star  & \star  \\
0 & 0  & A_3 & \star  & \star  & \star &  \star  & \star \\
0 & 0  & 0 & \mu I & 0 & 0   & \star &  \star  \\ 
0 & 0  & 0 & 0 & \mu I & 0   & \star &  \star   \\ 
0 & 0  & 0 & 0 & 0 & \mu I   & \star &  \star  \\ 
0 & 0  & 0 & 0 & 0 & 0   & A_7  &  \star \\ 
0 & 0  & 0 & 0 & 0 & 0   & 0  &  A_8  \\
\end{pmatrix}\end{equation}

 This means that $\mathcal K_{\mathcal E}$  is isomorphic to the determinant $1$ subgroup of a direct product of general linear groups. Note that  the conjugation action of every element of $R$ on $\mathcal K_{\mathcal E}$  is equal to a conjugation by  an element of  $\mathcal K_{\mathcal E}$, since $R$  is generated by $\mathcal K_{\mathcal E}$ and its centralizer. So, every IRS of $\mathcal K_{\mathcal E}$ is an IRS of $R$. 

\vspace{5mm}

%
%
%
%

\subsection{Plan of the paper}
 The paper is organized as follows. In \S \ref{prelims}, we establish some preliminary results:   we introduce in \S \ref{cocycle} a useful co-cycle associated to an IRS in $A \rtimes \Gamma$, prove two facts about finite measure preserving linear actions in \S \ref{linearactions}, and prove the result about transverse IRSs in \S \ref{transversesec}. Section \ref{parabolicsec} concerns IRSs  of parabolic subgroups, and in \S \ref{affinesec} we prove Theorems \ref{thm:slr-irs} and \ref{thm:affine-irss}.

\subsection{Acknowledgments}

We thanks the referee for a careful reading of the paper, a number of useful comments, and the suggestion to combine our work with \cite{bekka2007operator} to give a proof of the Nevo--Stuck--Zimmer theorem for $\SL_n(\Z)$, as described in Remark \ref{rmk} above.

\section{IRSs in general semidirect products}\label{prelims}
In this section we study semidirect products $G = A \rtimes \Gamma$, where $\Gamma$ acts on $A$
by automorphisms. As above, $\pr$ is the natural projection
$G \to \Gamma$.

\subsection{The cocycle $S_H$}
\label {cocycle}

Let $H$ be a subgroup of $G$. For each $M \in \pr H$ let
\begin{align*}
  S_H(M) = \{v \in A\,:\,(v,M) \in H\}.
\end{align*}
Then $S_H(I) = H \cap A$ is a subgroup of $A$ where $I \in \Gamma$ denotes the identity element.

Let $(v,M), (w,N) \in H$. Then $(v,M) \cdot (w,N) = (v \cdot M w,M N) \in H$. It
follows that
\begin{align}
  \label{eq:S}
  S_H(M N) = S_H(M) \cdot M S_H(N),
\end{align}
where multiplication here denotes that of sets:
$B\cdot C = \{b\cdot c\,:\,b \in B, c\in C\}$.
\begin{claim}
  \label{clm:cocycle}
 If $M\in \pr H$, then $S_H(M)$ is a left coset of $S_H(I)$.
\end{claim}

Here, Claim \ref{clm:cocycle} and Equation \eqref{eq:S} say that $S_H$ is a \emph{cocycle} $S_H : \pr H \longrightarrow S_H(I) \backslash A$.

\begin{proof} Suppose $(v,M)$ and $(w,M)$ are elements of $H$. Then $$H \ni (v,M)\cdot (w,M)^{-1} = (v,M) \cdot (M^{-1}w^{-1},M^{-1}) = (v \cdot MM^{-1} w^{-1}, I)=(vw^{-1},I).$$
And if $(v,M)$ and $(x,I)$ are elements of $H$, we have $$H \ni (x,I) \cdot (v,M)  = (x\cdot Iv, M) = (xv, M). \qedhere$$
\end{proof}

We end this section with a useful observation. As we will apply it only when $A$ is abelian, we use additive  notation here. Let $(w,N)$ be an
arbitrary element of $G$, and let $(v,M) \in H$.  Then
$(v,M)^{(w,N)} = (N^{-1}v + N^{-1}(M-I)w, M^N) \in H^{(w,N)}$. (Here,
$a^b=b^{-1}ab$.) Hence
\begin{align}
  \label{eq:S-conj}
  S_{H^{(w,N)}}(M^N) = N^{-1} S_H(M) + N^{-1}(M-I)w.
\end{align}

\subsection{Group actions preserving finite measures}
\label {linearactions}

Here are four useful lemmas.

\begin {lem}\label{secondfurst}
	 Suppose that $G$ is a locally compact second countable group, and the induced action of $Z \leq \Aut(G)$ on  the space $\mathrm{Sub}_{G}$ preserves a finite measure $\mu$  that is supported on lattices. Then $Z$ preserves the Haar measure of $G$.
\end {lem}
\begin {proof}
 For some $n$, the set $\mathcal S$ of lattices with covolume in $[\frac 1n,n]$  has positive measure. If $Z $ does not preserve Haar measure $\nu$,  there is some $A \in Z$  with $A_* \nu = c \nu$ with $c> n^2$. The sets $A^{i  } \mathcal S$, where $i\in \Z$,  are then all disjoint and have the same positive measure.  This is a contradiction.
\end {proof}

The three following lemmas are inspired by an argument Furstenberg used  in his proof of the Borel density theorem \cite[Lemma 3]{Furstenbergnote}.

\begin {lem}\label{firstfurst}
Suppose a group $Z$ acts linearly on $\R^d$	 preserving a finite measure $m$, and $ V = \mathrm{Span}(\mathrm{supp} \, m)$. Then the image of the map $Z \longrightarrow \GL(V)$  is precompact.
\end {lem}

\begin{proof}
Restricting, it suffices to prove the lemma when $ \mathrm{Span}(\mathrm{supp} \, m)=\R^d$.  Let $(z_n)$ be a sequence in $Z$. After passing to a subsequence, we can assume that  there is some subspace $W \subset \R^d$  such that the maps $z_n |_W$  converge to some linear map $z : W \longrightarrow \R^d$, while $z_n(x) \to \infty$ if $x\in \R^d \setminus W$.  For instance, one can take $W$  to be any subspace  that is maximal among those for which there exists a subsequence $(z_{n_k})$ with  the property that $z_{n_k}(x)$ is bounded for all $x\in W$, and then  pass to a subsequence of such a subsequence.

If in the above, we always have $W = \R^d$, we are done. So, assume $W \neq \R^d$. Pick a metric inducing the one-point compactification topology on $\R^d \cup \infty$ and let $D : \R^d \cup \infty \longrightarrow \R$  be the distance to  the closed set $z(W) \cup \infty$.  By the dominated convergence theorem,
$$\int D(x) \, dm(x) = \int D(z_n(x)) \, dm(x) \longrightarrow 0,$$
 so $m$  is supported on $z(W)$. But as $W $ is a proper subspace, so is $z(W)$.  This contradicts our assumption that $ \mathrm{Span}(\mathrm{supp} \, m)=\R^d$.
\end{proof}

\begin {lem}\label {3rdfurst}
 Suppose that $\R^d=\oplus_i \mathcal L_i$, a  direct sum of subspaces, and that $\mu$ is a finite Borel measure on the Grassmannian of $k$-dimensional subspaces of $\R^d$. 
Suppose that for each $j$, there is a linear map $A_j : \R^d \longrightarrow \R^d$  that acts as a scalar map $v \mapsto \lambda_i v$  on each subspace $\mathcal L_i$, satisfies $\lambda_j > \lambda_i$  for $i\neq j$, and induces a map on the Grassmannian  that preserves $\mu$. Then $\mu$ is concentrated on  subspaces $W\subset \R^d$ that are direct sums of subspaces of the $\mathcal L_i$: $$W=\oplus_i S_i, \ \ S_i \subset \mathcal L_i.$$\end {lem}
\begin {proof}
The argument is similar to that of Lemma \ref {firstfurst}.  Denote the Grassmannian of $k$-subspaces of $\R^d$ by $Gr(k,d)$, fix $j$ and let $\mathcal Z_j$ be the closed subset of $Gr(k,d)$ consisting of all subspaces of the form $ S_j \oplus P'$,  where $S_j \subset \mathcal L_j$ and $P' \subset \oplus_{i\neq j} \mathcal L_i$. Given an element $P \in Gr(k,d)$, let $D_j(P)$ be the distance from $P$ to $\mathcal Z_j$, with respect to some metric inducing the natural topology. Then for each $P \in Gr(k,d)$,  we have $D((A_j)^n(P)) \to 0$ as $n\to \infty$. Hence, the dominated convergence theorem says that $$\int D(P) \, d\mu(P) = \int D((A_j)^n(P)) \, d\mu(P) \to 0.$$
So, $\mu$ is supported on $\mathcal Z_j$. This works for all $j$,  so the lemma follows.
\end {proof}

\begin{lem}
		Let $V,W$ be two vector spaces and let $\mathcal L(V,W)$ be the space of all linear maps from $V$ to $W$. Suppose that $X \subset \mathcal L(V,W)$ is a random subspace whose law is invariant under the action of $SL(V) \times SL(W)$. Then almost surely, $X$ is either $\{0\}$ or $\mathcal L(V,W)$.\label{4thfurst}
\end{lem}

Here, $(A,B) \in SL(V) \times SL(W)$ acts by sending $T\in \mathcal L(V,W)$ to $ATB^{-1}$.

\begin{proof}
 $SL(V) \times SL(W)$ is semisimple, and hence is a \emph{m.a.p.\ group}, in the sense of Furstenberg's paper  \cite{Furstenbergnote}. By \cite[Lemma 3]{Furstenbergnote} and the exterior power trick in the subsequent `Theorem', any finite $SL(V) \times SL(W)$-invariant measure on the set of subspaces of $\mathcal L(V,W)$ is supported on subspaces that are invariant under the $SL(V) \times SL(W)$ action.  But it is easy to check that the only such subspaces are the two trivial ones. \end{proof}

\begin{remark}
The proof of Lemma \ref{4thfurst} above is a bit silly since it relies on certain well-known facts, e.g.\ that semisimple groups are m.a.p., that are considerably harder to prove than Lemma \ref{4thfurst} itself. Really, one  can just prove the lemma by applying the arguments from Furstenberg's paper to certain well-chosen sequences of elements in $SL(V) \times SL(W)$. We encourage the reader to do this, while we lazily give the short proof above.
\end{remark}

\subsection{Transverse IRSs}\label {transversesec}
Let $A$ and $\Gamma$ be locally compact, second countable topological groups, and suppose $\Gamma$ acts by continuous automorphisms on $A $. Let $\Gamma_{triv}$ be the kernel of the action, and let $G=A \rtimes \Gamma$ be the associated semidirect product. 

We call a subgroup $H\leq G$ \emph {transverse} if $H\cap A=\{0\}$. For example, in the direct product $A \times A$,  the diagonal subgroup is transverse, as is the second factor.

\begin{thm}[Structure of transverse IRSs in semidirect products, part 1]
  \label{thm:general-irs1}
  Suppose $G=\R^d \rtimes \Gamma$ and $H$ is a transverse IRS of $G=\R^d \rtimes \Gamma$. Then  $\pr H \leq  \Gamma_{triv}$ almost surely.  
  \end{thm}
  
  \begin{remark}
Theorem~\ref{thm:general-irs1} also applies when $G=S \rtimes \Gamma$ and $S$ is a closed subgroup of $\R^d$. Indeed, the $\Gamma$-action on such an $S$ extends to the span of $S$ to which Theorem \ref{thm:general-irs1} applies, and any transverse IRS of $G=S \rtimes \Gamma$ induces a transverse IRS of $G = \rm{span}(S) \rtimes \Gamma$. 
\end{remark}

\begin{remark}
If the action $\Gamma \circlearrowleft A$ is faithful (as it
is, for example, in the case of the special affine groups), then Theorem~\ref{thm:general-irs1} implies there
are no nontrivial transverse IRSs of $G$.  Also, note that the theorem fails when $A$ is not torsion-free abelian. For instance, if $A$ is finite then a random conjugate of $\Gamma$ is an IRS of $A \rtimes \Gamma$. And if $A$ is not abelian, the antidiagonal $$\{(g,g^{-1}) \ | \ g \in A\} \subset A \rtimes A,$$
where $a\in A $ acts on $x\in  A$ by $a(x)=a^{-1}xa$, is a normal
subgroup of $A \rtimes A$  that does not  project into
$A_{triv}=Z(A)$. However, we expect that for general $A$, if $H$
is a transverse IRS of $A \rtimes\Gamma$, then the action of any
element of $\pr H$ on $A$ is well-approximated by inner automorphisms
of $A$ in some sense. 
\end{remark}

\begin {proof}[Proof of Theorem \ref{thm:general-irs1}]
Let $H$  be a nontrivial transverse IRS of $G$.  In order to get a contradiction, suppose that it is not the case that $\pr H \leq  \Gamma_{triv}$ almost surely. Then there is an open subset $U \subset \Gamma$ with compact closure such that $U \cap  \Gamma_{triv} = \emptyset$, and $\pr H \cap U \neq \emptyset$ with positive probability.  In addition we choose $U$ small enough so that  for some $w\in \R^d $, some $0 < b_1 < b_2\in \R_+$ and some linear $L : \R^d \longrightarrow \R$,  we have that \begin{equation}
	 b_1 \leq L( (M-I)w ) \leq b_2, \text{ for all } M \in U.\label{defw}
\end{equation}

 Choose a left Haar measure $\mu_H$ on $\pr H$. By \cite[Claim A.2]{biringer2017unimodularity},  this can be done so that the $ \mu_H$  vary continuously  with $H \in \mathrm{Sub}_G $,  when regarded as measures on $\Gamma\geq \pr H$. 
 
 Because $H$ is transverse, $S_H(M)$ is a single element of $\R^d$ for any $M \in \pr H$. Selecting first a random $H \in \subg$ with $\pr H \cap U \neq \emptyset$, and then a $\mu_H$-random $M\in \pr H \cap U$, we can interpret the cocycle $S_H(M)$ as an  $\R^d$-valued random variable. Here, note that $\mu_H(\pr H \cap U)$  is always finite and nonzero, since $\pr H \cap U$  is nonempty, pre-compact and open in $H$.

%

Taking $w\in \R^d $  as in the first paragraph of the proof, let $H^{w} = (w,I)^{-1} H (w,I) $.  Since $ \pr H = \pr H^{w}$, we get a map $(H,M) \mapsto (H^w,M)$  defined on the domain \begin{equation} \big \{(H,M) \ | \ H \in \subg, \, \pr H \cap U \neq \emptyset,\,  M \in \pr H \cap U\big \} \label{domain}\end{equation}
 of the random variable $S_H(M)$.  As $H$ is an IRS,  this map is measure preserving, so the distributions of $S_{H^{w}}(M) $ and  $S_H(M)$ are equal, say to a probability measure $m_U $ on $\R^d$.

By~\eqref{eq:S-conj}, we have  $S_{H^w}(M) = S_H(M) + (M-I)w$  for all $M \in \pr H = \pr H^{w}.$ Iterating  the conjugation by $w$ and using \eqref{defw}, 
\begin{equation}
L \big ( S_H(M)  \big ) + nb_1 \leq L \big ( S_{H^{nw}}(M) \big )\leq L \big ( S_H(M)  \big ) + nb_2, \ \forall n\in \N. \label{Leqn}
\end{equation}
This contradicts  the fact that $m_U$  is a probability measure. For suppose $[a_1,a_2] \subset \R$  is an interval with $m_U ( L^{-1}([a_1,a_2]) )>0$. For a sufficiently sparse sequence $n_k\in \N$, the intervals
$[a_1+n_kb_1,a_2+n_kb_2] \subset \R$ are all disjoint. Hence, \begin{align*}1 \geq \sum_k m_U\big (L^{-1}[a_1+n_kb_1,a_2+n_kb_2]\big ) \geq \sum_k m_U(L^{-1}[a_1,a_2]) = \infty. \end{align*}
This contradiction proves the theorem.\end {proof}
%
%
%

\begin{thm}[Structure of transverse IRSs in semidirect products, part 2]
  \label{thm:general-irs2}
  Suppose $G=A \rtimes \Gamma$, $A$ is a simply connected nilpotent Lie group, $H$ is a transverse IRS of $G=A \rtimes \Gamma$ and $\lambda $ is the law of $H$. Let $$\mathcal H = \cup _{H \in \mathrm{supp} \, \lambda} H.$$  
If $\mathcal V \subseteq A $ is the Zariski closure of the set of first coordinates of all $(v,M) \in \mathcal H $,  then $\mathcal V$  is $\Gamma$-invariant  and the image  of the map $Z( \pr \mathcal H)\longrightarrow \Aut(\mathcal V)$ is precompact. 
\end{thm}
Here $Z( \pr \mathcal H)$ denotes the centralizer of $ \pr \mathcal H$ in $\Gamma$, and the Zariski closure of a subset of $A$ is the smallest connected Lie subgroup of $A$ containing that subset. 

\begin{remark}
Theorem~\ref{thm:general-irs2} also applies when $G=S \rtimes \Gamma$ and $S$ is a closed subgroup of some simply connected nilpotent Lie group $A$. Indeed, the $\Gamma$-action on such an $S$ extends to the Zariski closure $\overline S$ \cite[Theorem 2.11]{Raghunathandiscrete}, to which Theorem \ref{thm:general-irs2} applies, and any transverse IRS of $G=S \rtimes \Gamma$ induces a transverse IRS of $G = \overline S \rtimes \Gamma$. See \cite[Chapter II]{Raghunathandiscrete} for more information about the `Zariski closure' operation in simply connected nilpotent Lie groups, which behaves very similarly to `span' in $\R^d$.
\end{remark}

\begin{remark}
 To illustrate Theorem~\ref{thm:general-irs2}, suppose $A=\Gamma =\R^2$ and $(s,t) \in \Gamma$
acts by a rotation on $A$ with angle $s$. Then if
$$H_\theta = \Big\{ \big ( ( t \cos \theta, t \sin \theta ), (0,t) \big ) \ | \ t \in \R \Big \} \leq A \rtimes \Gamma,$$
 we obtain a transverse IRS of $G=A \rtimes \Gamma$ by randomly picking $\theta \in [0,2\pi]$  against Lebesgue measure. Here,  the centralizer $Z( \pr \mathcal H)$  is all of $\Gamma $, which acts compactly on $A$.

\end{remark}

\begin {proof}[Proof of Theorem \ref{thm:general-irs2}]
 The $\Gamma $-invariance  of $\mathcal V$  is  immediate.  For if $N \in \Gamma$  and $(v,M)\in \mathcal H$, \begin{equation}
(e,N)^{-1} (v,M) (e,N)= (N^{-1}v, N^{-1}MN). 
\label{conj2}	
 \end{equation}
Here, we write $e$ for the identity element since $A$ is not necessarily abelian. As $\supp \, \lambda $ is conjugation invariant,  the set of all $v\in A$  such that $(v,M)\in \mathcal H $  for some $M $ is $\Gamma$-invariant. Hence, its Zariski closure $\mathcal V$ is also $\Gamma$-invariant.

 As in the proof of Theorem \ref{thm:general-irs1}, choose $U \subset \Gamma$ with compact closure such that $\pr H \cap U \neq \emptyset$ with positive probability. Let $N\in Z (\pr \mathcal H)$ and write $H^N = (e,N)^{-1}H (e,N)$.  Substituting $N^{-1}MN=M$ in \eqref{conj2} we see that $\pr H =\pr H^N$, so as before the distribution of $S_{H^N}(M)$ is the same as $m_U$,  the distribution of $S_{H}(M)$. Now, though, \eqref{conj2} implies that $$S_{H^N}(M)=N^{-1}(S_{H}(M)).$$
So, the measure $m_U$ on $A$ is $Z (\pr \mathcal H)$-invariant.  

Since $A$ is a simply connected nilpotent Lie group, there is a diffeomorphism $\log : A \longrightarrow \mathfrak a$ to the Lie algebra $\mathfrak a$ that is an inverse for the Lie group exponential map \cite[1.127]{knapp2013lie}. Then $\log_* m_U$ is a probability measure on $\mathfrak a$ that is invariant under the induced action of $Z (\pr \mathcal H)$ on $\mathfrak a$. By Lemma \ref{firstfurst}, $Z (\pr \mathcal H)$ acts precompactly on the span $V_U = \mathrm{Span}(\mathrm{supp} \, \log_* m_U)$, and therefore it acts precompactly on the sum $V$ of all $V_U$, as $U$ ranges over all possible choices. But the Zariski closure $\mathcal V = \exp(V)$, so then $Z (\pr \mathcal H)$ acts precompactly on $\mathcal V$ as well.
%
\end {proof}

We present an easy corollary of Theorem \ref{thm:general-irs1}:
\begin{cor}
  \label{cor:RR}
   The only ergodic IRSs of the affine group $\R \rtimes \R^+$ are the point
  masses on its  closed, normal subgroups: $\{e\}, \, \R ,\,\R \rtimes \R^+$ and $ \R \rtimes \{\alpha^n \ | \ n \in \Z\}, $  where $\alpha>0$.
\end{cor}
Note that this stands in contrast to other metabelian groups (e.g.,
lamplighter groups) that have a rich set of invariant random
subgroups~\cite{bowen2015invariant}.
\begin{proof}[Proof of Corollary~\ref{cor:RR}]
  Let $H$ be a non-trivial ergodic IRS of $\R \rtimes \R^+$. If $H$  is transverse, then $\pr H =\{1\} \in \R_+$, by Theorem \ref{thm:general-irs1}. Hence $H=\{e\}$.

 Otherwise, the random subgroup $H \cap \R \subset \R$ is nontrivial almost surely, and its law is
  invariant  under the $\R^+$ action (i.e., multiplication by a scalar). So, $H \cap \R = \R$  almost surely, and $H = \R \rtimes \pr H$. But $\pr H$ is an ergodic IRS of $\R^+$, and thus must be a
  point mass on either $\{1\}$, $\R^+$ or $ \R \rtimes \{\alpha^n \ | \ n \in \Z\}, $  where $\alpha>0$. We have thus proved the claim.
\end{proof}

\section{IRSs of parabolic subgroups}
\label {parabolicsec}

 To recap our notation: 
$W=S_1 \oplus \cdots \oplus S_n$  is a real vector space, $\mathcal{F}$  is the associated flag $$0=W_0 < W_1 < \cdots < W_n = W, \ \ \ W_k = \oplus_{i=1}^k S_i,$$
 $P < \SL(W)$ is the  parabolic subgroup  stabilizing $\mathcal F$,   $V < P$    is the unipotent subgroup of all $A \in P$  that act trivially on each of the factors $W_i/W_{i-1}$, and 
$$P = V \rtimes R, \ \ R = \left \{ (A_1,\ldots,A_n) \in \prod_{i=1}^n \GL(S_i) \ \Big | \ \prod_i \det A_i = 1\right \}.$$
Also, $\mathcal E \subset \{1, \ldots, n\}^2$  will denote a subset of
pairs $(i,j)$ with $i<j$ that is closed
under `going up' and `going to the right', and we will let  $V_{\mathcal E}< P$ be
the normal subgroup consisting of all matrices that are equal to the
identity matrix except at entries corresponding to elements of
$\mathcal E$. Let $\mathcal K_{\mathcal E} < R$ be the kernel  of
the $R $-action (by conjugation) on $V/V_{\mathcal E}$.

 The goal of this section is to prove Theorem \ref{parabolictheorem}, i.e.\ that the  ergodic IRSs of $P$ are exactly the random subgroups of the form $V_{\mathcal E}\rtimes K$, where $K$ is an ergodic IRS of $\mathcal K_{\mathcal E}$.

\medskip

 We start with the following lemma.

\begin{lem}\label {intersectionLemma}
	 Suppose that $H$ is an invariant random subgroup of $P$  that lies in $V $. Then almost surely, $H = V_{\mathcal E}$ for some $\mathcal E$.
\end{lem}
\begin{proof}
Regard $V$ as the space of upper unitriangular block matrices, where the $ij^{th}$ entries is in $\mathcal L(S_j,S_i)$. It suffices to show that almost surely, $H $ is a `matrix entry subgroup', i.e.\ a subgroup determined by prescribing that some fixed subset of the matrix entries are all zero. As there are only finitely many such subgroups, it will follow that almost surely, $H$ is a matrix entry subgroup of $V$ that is a normal subgroup of $P$. A quick computation with elementary matrices shows that the only such subgroups are the $V_\mathcal E$ described above.

Let $H_0$ and $\overline H$  be the identity component and Zariski closure of $H$, respectively,  recalling that the \emph{Zariski closure} of a subgroup is the smallest connected Lie subgroup of $V$ containing it. (See \cite[Chapter II]{Raghunathandiscrete}.) Then $H_0$ and $\overline H$ are both $R$-invariant random subgroups of $V$.  Let $\mathfrak h_0$ and $\overline {\mathfrak h}$  be the associated Lie algebras, which are $R$-invariant random subspaces of the Lie algebra $\mathfrak v$ of $V$. One can identify $\mathfrak v$ with the set of all strictly upper triangular block matrices, where the $ij^{th}$ entry is an element of $\mathcal L(S_j,S_i)$. If we identify $\mathcal L(S_j,S_i)$ with the subspace of $\mathfrak v$ consisting of matrices that are nonzero at most in the $ij^{th}$ entry, then
$$\mathfrak v = \oplus_{i<j}\mathcal L(S_j,S_i).$$

The action $R \circlearrowright \mathfrak v$  leaves all the factors $\mathcal L(S_j,S_i)$ invariant. Given $k<l$, let $A $ be the matrix that has a $2I$ in the $kk^{th}$ entry and a $\frac 12 I$ in the $ll^{th}$ entry, and is otherwise equal to the identity matrix. Then the matrix $\frac 1{\det A } A $ lies in $ R$, and acts by conjugation on each $\mathcal L(S_j,S_i)$ as the scalar matrix $\frac {\lambda_{ij}} {\det A } I$ , where 
\begin {equation}\lambda_{ij} = \begin{cases} 4 &  (i,j)=(k,l) \\  2 & i=k,
    j\neq l \text{ or } j=l,i\neq k \\ \frac 12 & i=l  \text{ or }
    j=k, \text{ and } i \neq j\\ 1 & \text{ otherwise.} \end{cases}\label {eigs}\end{equation}
Applying Lemma \ref{3rdfurst} to the direct sum $\mathfrak v = \oplus_{i<j}\mathcal L(S_j,S_i)$, considered together with the actions of all the matrices $A$ obtained by varying $k,l$, we see that almost surely, both $\mathfrak h_0$  and $\overline {\mathfrak h}$ are direct sums of subspaces of the factors $\mathcal L(S_j,S_i)$. However, the only $R$-invariant random subspaces of a \emph{fixed} factor $\mathcal L(S_j,S_i)$ are the zero subspace and the entire $\mathcal L(S_j,S_i)$. (This follows immediately from Lemma \ref{4thfurst}, since one can embed  $SL(S_i) \times SL(S_j) \hookrightarrow R$ by taking $(A,B)$ to the element of $R$ that has $A \in \mathcal L(S_i,S_i)$ in the $ii$ entry and $B \in \mathcal L(S_j,S_j)$ in the $jj$ entry, and is otherwise equal to the identity matrix.) Hence, $\mathfrak h_0$  and $\overline {\mathfrak h}$ are almost always direct sums of the factors $\mathcal L(S_j,S_i)$ themselves, rather than subspaces thereof. 
In other words, $H_0$ and $\overline H$ are matrix entry subgroups almost surely.

Now $H_0 \subset H \subset \overline H$, so if $H_0=\overline H$, then $H $ is a matrix entry subgroup as desired. So, after restricting the law of $H$, we may assume that almost surely $H_0$ and $\overline H$ are \emph{fixed} matrix entry subgroups and that $H_0 \subsetneq H$. As $H$ is an IRS of $P$, $H_0$ is a normal subgroup of $P$.   We can then project $H$ to a $P$-invariant random subgroup $H/H_0$ of the quotient group $\overline H/H_0$. Since $V$ is a  nilpotent Lie group, the sub-quotient group $\overline H/H_0$ is as well. Every Zariski dense subgroup of a nilpotent Lie group is a lattice (c.f.\ \cite[Theorem 2.3]{Raghunathandiscrete}), so the $P$-invariant random subgroup $H/H_0 < \overline H/H_0$ is a lattice almost surely. Lemma \ref{secondfurst} then implies that the $P$ action on $\overline H/H_0$ preserves Haar measure. 

But if $\mathcal D$ is the set of matrix entries that in $\overline H $ are free to take on any value, and in $H_0$ are prescribed to be zero, there is a diffeomorphism
$$ \overline H / H_0 \longrightarrow \oplus_{(i,j)\in \mathcal D} \mathcal L(S_j,S_i)$$ that takes a matrix in $\overline H$ to the list of its $\mathcal D$-entries. If Lebesgue measures are chosen on the Euclidean spaces $\mathcal L(S_j,S_i)$, the resulting product measure pulls back to a Haar measure on $\overline H/H_0 $. So, one can witness that the action $R \circlearrowright \overline H/H_0$ does not preserve Haar measure as follows. Let $i_{\min}$ be the minimum $i$ such that there is some $(i,j) \in \mathcal D$, and $i_{\max}$ be the maximum $i$  such that there is some $(j,i) \in \mathcal D$, and define $A\in R$ by letting
$$A_{ii}=\begin {cases}
	 2^{1/\dim(S_{i_{min}})} I & i = i_{\min} \\ 
2^{-1/\dim(S_{i_{max}})} I & i = i_{\max} \\
	I & \text{otherwise}.
\end {cases}$$
This $A$ acts diagonally on $\oplus_{(i,j)\in \mathcal D} \mathcal L(S_i,S_j)$, and the action is scalar in each factor. Moreover, there are no entries of $\mathcal D$  directly above the $i_{\min}$ diagonal entry, and no entries to the right of the $i_{\max}$ diagonal entry, so  the eigenvalues of the action of $A$ on $\oplus_{(i,j)\in \mathcal D} \mathcal L(S_i,S_j)$ are $1$, $ 2^{1/\dim(S_{i_{min}})}$ an $ 2^{1/\dim(S_{i_{max}})}$. Hence, $A$ cannot preserve Lebesgue measure.  \end{proof}

 Now suppose that $H $ is an ergodic IRS of  $P=V \rtimes R$.  Lemma \ref{intersectionLemma} implies that there is some $\mathcal E$ such that $H \cap V=V_{\mathcal E}$ almost surely.  Applying Theorem \ref{thm:general-irs1} to the transverse IRS that is the projection of $H$ to $(V/V_{\mathcal E})^{ab} \rtimes R$, where $( \, \cdot \, )^{ab}$ is abelianization, we see that $\pr H \subset R$ almost surely acts trivially on $(V/V_{\mathcal E})^{ab}$. But if $\mathcal A$ is the set of super diagonal entries in our block matrices that do not lie in $\mathcal E$, there is an isomorphism
$$(V/V_{\mathcal E})^{ab} \longrightarrow \oplus_{(i,j) \in \mathcal A} \mathcal L(S_i,S_j) $$ that comes from taking a matrix in $V$ to its list of $\mathcal A$-entries. It follows that a matrix in $R$ acts trivially on $(V/V_{\mathcal E})^{ab}$ if and only if it acts trivially on $V/V_{\mathcal E}$: triviality of the $(V/V_{\mathcal E})^{ab}$-action is enough to force the conditions on diagonal entries indicated in the matrix \eqref{bigmatrix} from the introduction. Hence, $\pr H$ almost surely lies in the kernel $\mathcal K_{\mathcal E}$ of the $V/V_{\mathcal E}$-action as desired.

We now know that $H \cap V=V_{\mathcal E}$  and $\pr H \subset \mathcal K_{\mathcal E}$ almost surely. We would like to conclude that $H$ has the form $V_{\mathcal E} \rtimes K$ for some IRS $K < K_{\mathcal E}$. Note that this is not immediately obvious---the diagonal in $\R^2$ is a normal subgroup that intersects the first factor trivially, but does not split as a product of subgroups of the two factors.
By  Theorem~\ref{thm:general-irs2}, we know that the centralizer $Z(\pr H) \subset R$ acts precompactly on $\mathcal X \subset V/V_{\mathcal E}$, where $\mathcal X$ is the Zariski closure in $V/V_{\mathcal E}$ of the projections of all first coordinates of elements $(v,M) \in H$. If $\mathcal X = \{V_{\mathcal E}\}$, we are done, since then the first coordinates of all $(v,M)\in H$ lie in $V_{\mathcal E} = H \cap V$ and $H$ must have the form $V_{\mathcal E} \rtimes K$ for some IRS $K < \mathcal K_{\mathcal E}$. 

 So, we may assume that $\mathcal X \mathcal V_{\mathcal E} \supsetneq V_{\mathcal E}$. Picking a matrix $B$ in the difference, there is some entry $(i,j) \not \in \mathcal E$ in which $B$ is nonzero. The  centralizer $Z(\pr H)$ contains all elements of $R$ all of whose diagonal entries are scalars,  so in particular it contains the matrix whose eigenvalues $\lambda$ are listed in \eqref{eigs} above.  The action of this matrix on $B$ scales the $(i,j)$ entry by $4$, so $Z(\pr H)$ does not act pre-compactly on $\mathcal X $, and we have a contradiction.

\section{IRSs of special affine groups}
\label {affinesec}

 Using Theorems~\ref{thm:general-irs1} and \ref{thm:general-irs2}, it is now fairly easy to prove the results on IRSs of special affine groups stated in the introduction.

\begin{proof}[Proof of Theorem~\ref{thm:slr-irs}]
  Let $H$ be a  nontrivial ergodic IRS of $\R^d \rtimes \slr$.   Suppose that $H \cap \R^d=\{0\}$  almost surely. As the action $\slr \circlearrowright \R^d$ is faithful, Theorem~\ref{thm:general-irs1}  implies that $H$ is trivial.  So, $H \cap \R^d$ is almost surely some nontrivial  subgroup of $\R^d$.

In order to prove $H \cap \R^d$ is either a lattice or $\R^d$, it suffices to prove that the Zariski closure of $H \cap \R^d$ is almost surely $\R^d$. If not, we get for some $1 \leq k \leq d-1$, a $\slr$-invariant probability measure on the Grassmannian of $k$-dimensional subspaces of $\R^d$. In the terminology of Furstenberg \cite{Furstenbergnote}, $\slr$ is a m.a.p.\ group, so  this measure must be concentrated on $\slr$-invariant points. (Apply \cite[Lemma 3]{Furstenbergnote}  to the $k^{th}$ exterior power of $\R^d$.)  However, no nontrivial subspaces of $\R^d$ are $\slr$-invariant.

Now suppose $H \cap \R^d$ is a lattice (almost surely).  Let $\mu$ denote the law of $H$. By decomposing $\mu$ over the map $H \mapsto H \cap \R^d$, we can write $\mu = \int \mu_\Lambda~d\nu(\Lambda)$ where $\nu$ is the pushforward of $\mu$ under $H \mapsto H\cap \R^d$ and $\mu_\Lambda$ is concentrated on the set of subgroups $H$ such that $H\cap \R^d  = \Lambda$. By ergodicity $\nu$ is supported on the set of lattices of some fixed covolume $c>0$. Moreover $\nu$ is $\slr$ invariant since the map $H \mapsto H \cap \R^d$ is equivariant. Since $\slr$ acts transitively on this set of lattices, it follows that $\nu$ must be the Haar measure. 

By equivariance, we must have $\mu_{g\Lambda} = g_*\mu_\Lambda$ for $g \in \slr$ and $\nu$-a.e. $\Lambda$. Because $\slr$ acts transitively on the set of lattices with fixed covolume, we can assume without loss of generality that $\mu_{g\Lambda} = g_*\mu_\Lambda$ holds for every $g\in \slr$ and lattice $\Lambda$. 

We claim that $\mu_{\Lambda}$-a.e.\ $H$ is contained in
$\Lambda \rtimes \SL(\Lambda)$. First let $(v,M) \in H$. For any
$w\in \Lambda$ we have that $(w,I) \in H$, and so
$$(v,M)(w,I)(v,M)^{-1} = (M w, I) \in H \cap \R^d = \Lambda.$$ Because
$w\in \Lambda$ is arbitrary, $M \in \SL(\Lambda)$. Next observe that
the law of $H$ is invariant under conjugation by
$\Lambda \rtimes \SL(\Lambda)$. So if there exists
$M \in \SL(\Lambda)$ such that $S_H(M) \ne \Lambda$ with positive
probability then $M H M^{-1} \cap \R^d \ne \Lambda$ with positive
probability. This contradiction shows that $S_H(M)=\Lambda$ almost
surely which implies $H \le \Lambda \rtimes \SL(\Lambda)$. Thus
$\mu_\Lambda$ is the law of an IRS of $\Lambda \rtimes
\SL(\Lambda)$. This IRS must be ergodic because $\mu$ is ergodic.

%
\end{proof}

\begin {proof}[Proof of Theorem \ref{thm:affine-irss}] Let $H$ be a non-trivial, ergodic IRS of $G=\Z^d \rtimes \slz$. Then $H \cap \Z^d$ is a
random subgroup of $\Z^d$ whose law is invariant to the $\slz$
action. Note that since the action $\slz \circlearrowright \Z^d$  is faithful, Theorem \ref{thm:general-irs1}  implies that $H \cap \Z^d\neq \{0\}$.
Since there are only countably many subgroups of $\Z^d$, the
distribution of $H \cap \Z^d$ must be concentrated on a single, finite
$\slz$-orbit. So, $H\cap \Z^d$ is almost surely finite index in $\Z^d$. 

Let $O = \{M(H\cap \Z^d)\,:\,M \in \slz\}$ be the orbit of $H\cap \Z^d$ under
the $\slz$ action. Now, the intersection of the groups in this orbit
is also finite index in $\Z^d$, and is furthermore $\slz$-invariant,
and so must equal $n\Z^d$ for some $n \in \N$.

Recall that $G_n = (n \Z^d) \rtimes \Gamma(n)$, and let
$H_n = H \cap G_n$, a finite index subgroup of $H$. Using the  cocycle notation of \S \ref{cocycle}, for any
$M \in \pr H_n$ it holds that $S_H(M) = S_H(I):=H\cap \Z^d$, since otherwise
$S_H(M)$ is a non-trivial coset of $S_H(I)$, and its intersection with
$n\Z^d$, a subgroup of $S_H(I)$, is trivial, thus excluding $M$ from
$\pr H_n$. It follows that $H_n = (n\Z^d) \rtimes (\pr H_n)$.  This
completes the proof of Theorem~\ref{thm:affine-irss}.\end{proof}





\bibliographystyle{abbrv}
\bibliography{affine}

\begin{thebibliography}{10}

\bibitem{abert2012growth}
M.~Abert, N.~Bergeron, I.~Biringer, T.~Gelander, N.~Nikolov, J.~Raimbault, and
  I.~Samet.
\newblock On the growth of {$L^2$}-invariants for sequences of lattices in
  {L}ie groups.
\newblock {\em Ann. of Math. (2)}, 185(3):711--790, 2017.

\bibitem{Abertkesten}
M.~Ab\'ert, Y.~Glasner, and B.~Vir\'ag.
\newblock Kesten's theorem for invariant random subgroups.
\newblock {\em Duke Math. J.}, 163(3):465--488, 2014.

\bibitem{bader2014amenable}
U.~Bader, B.~Duchesne, J.~L{\'e}cureux, and P.~Wesolek.
\newblock Amenable invariant random subgroups.
\newblock {\em Israel Journal of Mathematics}, pages 1--24, 2014.

\bibitem{bekka2007operator}
B.~Bekka.
\newblock Operator-algebraic superridigity for sl n (z), n$\geq$3.
\newblock {\em Inventiones mathematicae}, 169(2):401--425, 2007.

\bibitem{Benedettilectures}
R.~Benedetti and C.~Petronio.
\newblock {\em Lectures on hyperbolic geometry}.
\newblock Universitext. Springer-Verlag, Berlin, 1992.

\bibitem{biringer2017unimodularity}
I.~Biringer and O.~Tamuz.
\newblock Unimodularity of invariant random subgroups.
\newblock {\em Trans. Amer. Math. Soc.}, 369(6):4043--4061, 2017.

\bibitem{bowen2010random}
L.~Bowen.
\newblock Random walks on random coset spaces with applications to furstenberg
  entropy.
\newblock {\em Inventiones Mathematicae}, 196(2):485--510, 2014.

\bibitem{bowen2012invariant}
L.~Bowen.
\newblock Invariant random subgroups of the free group.
\newblock {\em Groups Geom. Dyn.}, 9(3):891--916, 2015.

\bibitem{bowen2015invariant}
L.~Bowen, R.~Grigorchuk, and R.~Kravchenko.
\newblock Invariant random subgroups of lamplighter groups.
\newblock {\em Israel Journal of Mathematics}, 207(2):763--782, 2015.

\bibitem{Chabautylimite}
C.~Chabauty.
\newblock Limite d'ensembles et g\'eom\'etrie des nombres.
\newblock {\em Bull. Soc. Math. France}, 78:143--151, 1950.

\bibitem{Diattalattices}
A.~Diatta and B.~Foreman.
\newblock Lattices in contact {L}ie groups and 5-dimensional contact
  solvmanifolds.
\newblock {\em Kodai Math. J.}, 38(1):228--248, 2015.

\bibitem{Furstenbergnote}
H.~Furstenberg.
\newblock A note on {B}orel's density theorem.
\newblock {\em Proc. Amer. Math. Soc.}, 55(1):209--212, 1976.

\bibitem{hartman2015furstenberg}
Y.~Hartman and O.~Tamuz.
\newblock Furstenberg entropy realizations for virtually free groups and
  lamplighter groups.
\newblock {\em Journal d'Analyse Mathématique}, 126(1):227--257, 2015.

\bibitem{knapp2013lie}
A.~W. Knapp.
\newblock {\em Lie groups beyond an introduction}, volume 140.
\newblock Springer Science \& Business Media, 2013.

\bibitem{nevo2002generalization}
A.~Nevo and R.~J. Zimmer.
\newblock A generalization of the intermediate factors theorem.
\newblock {\em Journal d'Analyse Math{\'e}matique}, 86(1):93--104, 2002.

\bibitem{Phelpslectures}
R.~R. Phelps.
\newblock {\em Lectures on {C}hoquet's theorem}, volume 1757 of {\em Lecture
  Notes in Mathematics}.
\newblock Springer-Verlag, Berlin, second edition, 2001.

\bibitem{Raghunathandiscrete}
M.~S. Raghunathan.
\newblock {\em Discrete subgroups of {L}ie groups}.
\newblock Springer-Verlag, New York, 1972.
\newblock Ergebnisse der Mathematik und ihrer Grenzgebiete, Band 68.

\bibitem{stuck1994stabilizers}
G.~Stuck and R.~J. Zimmer.
\newblock Stabilizers for ergodic actions of higher rank semisimple groups.
\newblock {\em Annals of Mathematics}, pages 723--747, 1994.

\bibitem{thomas2014invariant}
S.~Thomas and R.~Tucker-Drob.
\newblock Invariant random subgroups of strictly diagonal limits of finite
  symmetric groups.
\newblock {\em Bulletin of the London Mathematical Society}, 46(5):1007--1020,
  2014.

\bibitem{tits1976systemes}
J.~Tits.
\newblock Systemes g{\'e}n{\'e}rateurs de groupes de congruence.
\newblock {\em CR Acad. Sci. Paris S{\'e}r. AB}, 283(9), 1976.

\bibitem{tucker2015weak}
R.~D. Tucker-Drob.
\newblock Weak equivalence and non-classifiability of measure preserving
  actions.
\newblock {\em Ergodic Theory and Dynamical Systems}, 35(01):293--336, 2015.

\bibitem{Vershiktotally}
A.~M. Vershik.
\newblock Totally nonfree actions and the infinite symmetric group.
\newblock {\em Mosc. Math. J.}, 12(1):193--212, 216, 2012.

\end{thebibliography}

\end{document}